\documentclass[a4paper,leqno]{amsart}

\usepackage{amsbsy,amscd,amsmath,amsopn,amssymb,amsthm}

\theoremstyle{definition}
\newtheorem{thm}{Theorem}[section]
\newtheorem{lem}[thm]{Lemma}
\newtheorem{prp}[thm]{Proposition}
\newtheorem{dfn}[thm]{Definition}

\newtheorem{rmk}[thm]{Remark}

\newcommand{\R}{{\mathbf{R}}}
\newcommand{\C}{{\mathbf{C}}}
\newcommand{\N}{{\mathbf{N}}}

\begin{document}
\title[]{Structure of the spaces of matrix monotone functions and 
of matrix convex functions and Jensen's type inequality for 
operators}

\author{HIROYUKI OSAKA}
\thanks{}
\address{Department of Mathematical Sciences, Ritsumeikan
University, Kusatsu, Shiga 525-8577, Japan}

\email{osaka@se.ritsumei.ac.jp}

\author{JUN TOMIYAMA}
\thanks{}
\address{Prof. Emeritus of Tokyo Metropolitan University,
201 11-10 Nakane 1-chome, Meguro-ku, Tokyo, Japan}

\email{jtomiyama@fc.jwu.ac.jp}

\footnote{Research partial supported by Ritsumeikan Rsearch Proposal 
Grant, Ritsumeikan University 2007-2008.}

\maketitle

\begin{abstract}
Let $n \in \N$ and $M_n$ be the algebra of $n \times n$ matrices. 
We call a function $f$ matrix monotone of order $n$ or 
$n$-monotone in short whenever the inequality $f(a) \leq f(b)$ holds 
for every pair of selfadjoint matrices $a, b \in M_n$ such that 
$a \leq b$ and all eigenvalues of $a$ and $b$ are contained in $I$.
Matrix convex (concave) functions on $I$ are 
similarily defined. 
The spaces for $n$-monotone functions and $n$-convex functions 
are written as $P_n(I)$ and $K_n(I)$. We also denote that 
$P_n^+(I) = \{f \in P_n(I)\colon f(I^\circ) \subset (0, \infty)\}$, where 
$I^\circ$ means the set of inner points in $I$.

In this note we charactrize $n$-monotone functions and $n$-convex functions from the point of Jensen's type inequlity for operators.
For each $n \in \N$ and a finite interval $I$ we define the class $C_n(I)$
by the set of all positive real-valued continuous functions $f$ over $I$ 
such that $f(I^\circ) \subset (0, \infty)$ and for any $n$-subset
$S \subset I^\circ$ 
there exists a positive Pick function $h$ on $(0, \infty)$ 
interpolating $f$ on $S$. 
Then we characterize $C_n([0, 1))$ by an operator inequality. 
Moreover we show that for each $n$ 
$C_{2n}([0, \infty)) \subsetneq P_n^+([0, \infty))$.

We also discuss several assertions at each leven $n$ for which we regard them
as the problems of double piling structure of those sequences $\{P_n(I)\}_{n\in\N}$ and $\{K_n(I)\}_{n\in\N}$. 
In order
to see clear insight of the aspect of the problems, however, we choose the following three main assertions
among them and discuss their mutual dependence:

\begin{enumerate}
\item[(i)] $f(0)\leq 0$ and $f$ is $n$-convex in $[0,\alpha)$,
\item[(ii)] For each matrix $a$ with its spectrum in $[0,\alpha)$ and a contraction $c$ in the matrix algebra
$M_n$,
 \[
f(c^{\star}a c)\leq c^{\star}f(a)c,
\]
\item[(iii)] The functon $g(t)/t$ is $n$-monotone in $(0,\alpha)$.
\end{enumerate}

In particular, we show that for any $n \in \N$ two conditions $(ii)$ and $(iii)$ are equivalent.
\end{abstract}

\section{Introduction}
Let $I$ be nontrivial interval of the real line $\R$ 
(open, closed, half-open etc.). 
A real valued continuous funtion $f$ on $I$ is said to be operator 
monotone if for every selfadjoint operators $a, b$ on a Hilbert space $H$ 
$(\dim H = + \infty)$ sych that $a \leq  b$ and 
$\sigma(a), \sigma(b) \subseteq I$  we have $f(a) \leq f(b)$. 

Let $n \in \N$ and $M_n$ be the algebra of $n \times n$ matrices. 
We call a function $f$ matrix monotone of order $n$ or 
$n$-monotone in short whenever the inequality $f(a) \leq f(b)$ holds 
for every pair of selfadjoint matrices $a, b \in M_n$ such that 
$a \leq b$ and all eigenvalues of $a$ and $b$ are contained in $I$.
Matrix convex (concave) functions on $I$ are 
similarily defined as above as well as 
operator convex (concave) functions. 
We denote the spaces of operator monotone functions 
and of operator convex functions by $P_\infty(I)$ and $K_\infty(I)$ 
respectively. The spaces for $n$-monotone functions and $n$-convex functions 
are written as $P_n(I)$ and $K_n(I)$. We also denote that 
$P_n^+(I) = \{f \in P_n(I)\colon f(I^\circ) \subset (0, \infty)\}$, where 
$I^\circ$ means the set of inner points in $I$.
We note that $P_{n+1}(I) \subseteq P_n(I)$ 
and $\cap_{n=1}^\infty P_n(I) = P_\infty(I)$.
Similarily, we have $K_{n+1}(I) \subseteq K_n(I)$ and 
$\cap_{n=1}^\infty K_n(I) = K_\infty(I)$. 

The first question is whether $P_{n+1}(I)$ $($resp. $K_{n+1}(I))$ is 
strictly contained in
$P_n(I)$ $($resp. $K_n(I))$ for every $n$. 
Although most of literatures assert the existence of such gaps, 
no explicit example was given in case $n \geq 3$ in spite of 
the longtime since the paper \cite{L} of Loewner in 1934. 
In \cite{HGT} Hansen, Ji and Tomiyama presented 
an explicit example with the gap between $P_{n+1}(I)$ and $P_n(I)$ for every $n$ and an  
interval $I$.
More general discussions are treated in \cite{OST} by 
Osaka, Silvestrov and Tomiyama about gaps of $\{P_n(I)\}_{n\in \N}$ and  
we have now abundant examples of polynomials in $P_n(I)\backslash P_{n+1}(I)$ 
using the trancated momonent problems for Hankel matrices in \cite{CF} of Curto  and Fialkow, 
In  \cite{HT} 
Hansen and Tomiyama also discussed about gaps of $\{K_n(I)\}_{n\in\N}$, and 
constructed abundant examples of polynomials  in $K_n(I)\backslash K_{n+1}(I)$.

On the contrary, in \cite{AKS}  Ameur, Kaijser and Silvestrov studied subclass $C_n(0, \infty)$ of interpolation functions 
of order $n$ of $P_n^+(0, \infty)$ and showed by a theorem of Doughue \cite{Dg}
that $C_n(0,\infty)$ coinsides with the class of functions such that
for each $n$-subset $S = \{\lambda_i\}_{i=1}^n$ there exists a positive Pick function $h$ on $(0, \infty)$ 
interpolating $f$ on $S$, that is, $h(\lambda_i) = f(\lambda_i)$ for each $1 \leq i \leq n$. 
They also showed that $P_2^+(0,\infty) \subsetneq C_3(0,\infty)$ and 
$C_4(0,\infty) \subsetneq P_2^+(0,\infty)$. 
{We recall that a complex analytic  function $h$ defined on $\{z \in \C\colon \Im(z) > 0\}$ is called a 
Pick function if their range is in the closed upper half plane $\{z \in \C\colon \Im(z) \geq 0\}$.

In this note we charactrize $n$-monotone functions and $n$-convex functions from the point of Jensen's type inequlity for operators.
For each $n \in \N$ and a finite interval $I$ we define the class $C_n(I)$
by the set of all positive real-valued continuous functions $f$ over $I$ 
such that $f(I^\circ) \subset (0, \infty)$ and for any $n$-subset
$S \subset I^\circ$ 
there exists a positive Pick function $h$ on $(0, \infty)$ 
interpolating $f$ on $S$. 
Then we characterize $C_n([0, 1))$ by an operator inequality. 
Moreover we show that for each $n$ 
$C_{2n}([0, \infty)) \subsetneq P_n^+([0, \infty))$.
This is an answer to a question in \cite{AKS}.

Let $0 < \alpha \leq \infty$.There is then a well known series of equivalent assertions connecting
operator convex functions in the interval $[0,\alpha)$ and operator monotone functions in the interval
$[0,\alpha)$ including Jensen's type inequality (cf.\cite{HP}). 

In section 3 we shall discuss those ( equivalent) assertions at each leven $n$ for which we regard them
as the problems of double piling structure of those sequences $\{P_n(I)\}$ and $\{K_n(I)\}$. In order
to see clear insight of the aspect of the problems, however, we choose the following three main assertions
among them and discuss their mutual dependence:

\begin{enumerate}
\item[(i)] $f(0)\leq 0$ and $f$ is $n$-convex in $[0,\alpha)$,
\item[(ii)] For each matrix $a$ with its spectrum in $[0,\alpha)$ and a contraction $c$ in the matrix algebra
$M_n$,
 \[f(c^{\star}a c)\leq c^{\star}f(a)c,\]
\item[(iii)] The functon $g(t)/t$ is $n$-monotone in $(0,\alpha)$.
\end{enumerate}

We study these three conditions in the classes of matrix convex functions and
matrix monotone functions and show that for each $n$ the condition $(ii)$ 
is equivalent to the condition $(iii)$. The assertion that 
 $f$ is $n$-convex with $f(0) \leq 0$ implies that $g(t)$ is 
 $(n-1)$-monotone holds, however, the converse does not hold even if 
 $n = 1$. We do not know that the assertion 
 $f$ is $n$-convex with $f(0) \leq 0$ implies that $g(t)$ is  
 $n$-monotone, but it holds in the case that $f$ is a 2-matrix convex polynomial of 
 the degree not greater than 5 with $f(0) \leq 0$.
 
The authors would like to thank Dr. Yacin Ameur for a fruitful discussion 
about interpolation class $C_n(0, \infty)$ and Professor Sergei Silvestrov for 
hearty hospitality when they stayed at Lund Univ. in May, 2006.
 
\section{The class $C_n$}

\begin{dfn}
Let $I$ be a finite interval (open, closed, or open-closed).
For $n \in \N$ we denote $C_n(I)$ be the set of all positive 
real-valued continuous interpolation functions $f$ over $I$ such that 
for any $\{\lambda_i\}_{i=1}^n \subset I^\circ$ there is a 
Pick function $h\colon (0, 1) \rightarrow \R $ such that 
$f(\lambda_i) = h(\lambda_i)$ for $1 \leq i \leq n$, 
where $I^\circ$ denotes the set of inner points in $I$.
\end{dfn}

For two finite intervals of the same type such as an open, 
half-open like $[\alpha, \beta)$ and $[\gamma, \delta)$ 
one can easily find an monotone increasing linear function 
$h\colon [\gamma, \delta) \rightarrow [\gamma, \delta)$ with the 
inverse  function $h^{-1}\colon [\gamma, \delta) \rightarrow 
[\alpha, \beta)$ having the same property. As both fucntions 
$h$ and $h^{-1}$  are operator monotone and operator convex 
fnctions the set $C_n([\alpha, \beta))$  and $C_n([\gamma, \delta))$
is easily transfered each other. So we consider the case that 
$I = [0,1]$. 

The following is the characterization of a class $C_n([0, 1])$

\begin{thm}
Let $f\colon [0, 1] \rightarrow \R$ be a continuous function.
The followings are equivalent.

\begin{enumerate}
\item[$(1)$]
$f \in C_n([0, 1])$.
\item[$(2)$]
For any $\{\lambda_i\}_{i=1}^n \subset (0, 1)$ if 
$$
\sum_{i=1}^na_i\frac{(1+t)\lambda_i}{1 + (t - 1)\lambda_i} \geq 0
$$
for 
any $\{a_i\}_{i=1}^n \subset \R$ we have 
$$
\sum_{i=1}^na_if(\lambda_i) \geq 0.
$$
\item[$(3)$]
For any $A, T \in M_n(\C)$ with $T^*T \leq 1$ and $\sigma(A) \subset (0, 1)$
$$
T^*AT \leq A \Longrightarrow T^*f(A)T \leq f(A).
$$
\end{enumerate}
\end{thm}

\begin{proof}
$(1) \rightarrow (3)$:

Take $T, A \in M_n(\C)$ satisfying 
$T^*T \leq 1$ and $\sigma(A) \subset (0, 1)$.
Set $\phi\colon (0,1) \rightarrow (0, \infty)$ by 
$\phi(t) = \frac{t}{1 - t}$. 
Then $\phi$ is operator monotone. 
Hence 
$T^*\phi(A)T \leq \phi(A)$ by \cite{Hs}.

Since $f \circ \phi^{-1}\colon (0, \infty) \rightarrow \R \in C_n((0, \infty))$,by \cite[Corollary 2.4]{AKS}
we have 
$$
T^*((f\circ \phi^{-1})(\phi(A))T \leq (f \circ \phi^{-1})(\phi(A)),
$$
and 
$T^*f(A)T \leq f(A)$.

$(3) \rightarrow (1)$:

Take $A, T \in M_n(\C)$ with $T^*T \leq 1$ and $\sigma(A) \subset (0, \infty)$.
Since $\phi^{-1}\colon (0, \infty)$ is operator monotone and 
$\sigma(A) \subset (0, \infty)$, from \cite{Hs} we have 
$$
T^*\phi^{-1}(A)T \leq \phi^{-1}(A).
$$
Note that $\sigma(\phi^{-1}(A)) \subset (0, 1)$. 
Then from the assumption for $f$ we have
\begin{align*}
T^*f(\phi^{-1}(A))T &\leq f(\phi^{-1})(A)\\
T^*(f \circ \phi^{-1})(A)T &\leq (f \circ \phi^{-1})(A).\\
\end{align*}
Hence $f \circ \phi^{-1} \in C_n((0, \infty))$ 
from the definition, and we know 
$f \in C_n([0, 1])$ from \cite[Corollary 2.4]{AKS} and the definition.

$(1) \rightarrow (2)$:

Let $h$ be a Pick funtion on $(0, 1)$. Then $h \circ \phi^{-1}$ is one 
on $(0, \infty)$. Then since there is a positive Radon measure 
$\rho$ on $[0, \infty]$
such that 
$$
h \circ \phi^{-1} (\lambda) = \int_{[0,\infty]}\frac{(1+t)\lambda}{1 + t\lambda}
d\rho, \ \lambda > 0,
$$
we have 
$$
h(\lambda) = \int_{[0,\infty]}\frac{(1+t)\lambda}{1 + (t-1)\lambda}
d\rho, \ \lambda \in (0, 1).
$$
For $\{\lambda_i\}_{i=1}^n \subset (0, 1)$ suppose that 
$$
\sum_{i=1}^na_i\frac{(1+t)\lambda_i}{1 + (t - 1)\lambda_i} \geq 0
$$
for 
any $\{a_i\}_{i=1}^n \subset \R$. 
Since there is a Pick function on $(0, 1)$ such that 
$f(\lambda_i) = h(\lambda_i)$ for $1 \leq i \leq n$,
\begin{align*}
\sum_{i=1}^na_if(\lambda_i) 
&= \sum_{i=1}^n\int_{[0,\infty]}a_i\frac{(1+t)\lambda_i}{1 + (t-1)\lambda_i}
d\rho\\
&= \int_{[0,\infty]}\sum_{i=1}^na_i\frac{(1+t)\lambda_i}{1 + (t-1)\lambda_i}
d\rho \geq 0
\end{align*}

$(2) \rightarrow (1)$:

Take $\{\lambda_i\}_{i=1}^n$ in $(0, 1)$ and fix them.
Set $A = C_\R[0,\infty]$ and 
$$
G = \{g\colon [0, \infty] \rightarrow \R\mid 
g(t) = \sum_{i=1}^na_i\frac{(1 + t)\lambda_i}{1 + (t -1 )\lambda_i}, 
\{a_i\}_{i=1}^n \subset \R\}.
$$
Here $A$ is a Banach space with respect to a norm 
$||k|| = \sup_{t\in [0,\infty]}|k(t)|$.

Then $G$ is a linear subspace of $A$. 
Let $\ell\colon G \rightarrow \R$ be a linear functional 
defined by
$$
\ell(\sum_{i=1}^na_i\frac{(1 + t)\lambda_i}{1 + (t -1 )\lambda_i})
= \sum_{i=1}^na_if(\lambda_i).
$$
Then $\ell$ is positive from the assumption.
Note that for any $\lambda \in (0, 1)$ we have 
$$
\min_{t\in [0,\infty]}\frac{(1+t)\lambda}{1 + (t - 1)\lambda} > 0.
$$
Take $c > 0$ such that $c\frac{(1+t)\lambda_1}{1 + (t - 1)\lambda_1} \geq 1$ 
and $t > 0$, and set 
$g_0(t) = c\frac{(1+t)\lambda_1}{1 + (t - 1)\lambda_1}$. 

Define $m \colon G \rightarrow \R$ by 
$m(g) = \sup\{g(t)\mid t \in [0, \infty]\}$. 

We will show that 
$$
\ell(g) \leq ||g + h||_{\ell(g_0)}, \ \forall h \in C_\R[0, \infty]_+,
$$
where $||k||_{\ell(g_0)} = ||h||\ell(g_0)$ and 
$C_\R[0, \infty]_+$ denotes a set of all positive functions 
in $C_\R[0, \infty]$.

For any $g \in G$ $m(g) < 0$ or $m(g) \geq 0$.
If $m(g) < 0$, $g(t) < 0$ for any $t \in [0, \infty]$, and
$$
\ell(g) < 0 \leq ||g + h||_{\ell(g_0)}, \ \forall h \in C_\R[0,\infty]_+.
$$
If $m(g) \geq 0$, we have 
\begin{align*}
g(t) &\leq m(g) \leq m(g)1 \\
&\leq m(g)g_0\\
\ell(g)
&\leq m(g + h)\ell(g_0) \\
&\leq ||g+h||\ell(g_0) = ||g + h||_{\ell(g_0)}, 
\ \forall h \in C_\R[0,\infty]_+.
\end{align*}

By Sparr's Theorem \cite[Lemma~2]{Sp} there is a positive linear functional 
$L\colon C_\R[0, \infty] \rightarrow \R$ such that
\begin{align*}
&L(k) \geq 0, \ \forall k \in C_\R[0,\infty]_+\\
&L(h) \leq ||h||_{\ell(g_0)}, \ \forall h \in C_\R[0, \infty].
\end{align*}
From the Riesz representation Theorem there is a positive Radon measure 
$\rho$ on $[0, \infty]$ suhc that 
$$
L(k) = \int_{[0, \infty]}k(t)d\rho(t), \ k \in C_R[0, \infty].
$$
Set $g_i(t) = \frac{(1+ t)\lambda_i}{1 + (t -1)\lambda_i}$ for
$1 \leq i \leq n$. Then we have 
\begin{align*}
f(\lambda_i) &= \ell(g_i)\\
&= L(g_i)\\
&= \int_{[0, \infty]}\frac{(1+ t)\lambda_i}{1 + (t -1)\lambda_i}d\rho(t)\\
&= h(\lambda_i)
\end{align*}
for $1 \leq i \leq n$ and a Pick function 
$$
h(\lambda) 
= \int_{[0, \infty]}\frac{(1+ t)\lambda}{1 + (t -1)\lambda}d\rho(t).
$$
\end{proof}

\vskip 5mm

The following is a partial answer to \cite[conjecture]{AKS}.

\begin{prp}
For each $n \in \N$ $C_{2n}([0, \infty) \subsetneq P_{n}^+([0,\infty))$.
\end{prp}

\begin{proof}
Take $n \in \N$ and consider a gap function $g_n \in P_{n}^+([0,\alpha_n])$
for some $\alpha_n > 0$ :
$$
g_n(x) = x + \frac{1}{3}x^3 + \cdots + \frac{1}{2n - 1}x^{2n-1}
$$

Suppose that $g_n \in C_{2n}([0, \alpha_n])$. 
Take a set $S \subset (0, \alpha_n)$ of $2n$ numbers and 
take a subset $S' \subset S$ with $|S'| = 2n -1$.
Since $g_n \in C_{2n}$, there is a Pick functiion of $\phi$ which 
are equal at points of $S$. Then $\phi$ and $g_n$ are equal at points of 
$S'$.

Then in \cite[XIV Theorem~3]{Dg} since $g_n$ does not satisfy 
condition $(i), (ii)$ (See \cite{HGT}.),  
$\phi$ and $g_n$  are equal only at points of $S'$. 
But this is a contradiction to the fact that $\phi$ and $g_n$ 
are equal at $S \supsetneq S'$.

Hence $g_n \notin C_{2n}([0, \alpha_n])$. 
Using an operator monotone function  
$h(t) = \frac{t}{\alpha_n - t}\colon [0, \alpha_n) \rightarrow [0,\infty)$.
We know that $g_n \circ h^{-1} \in P_n^+([0, \infty))$, but 
$g_n \circ h^{-1} \notin C_{2n}([0, \infty))$.

\end{proof}

\section{Double piling structure of matrix monotone functions and matrix convex functions }
Through this section we use symbols $a, b, \dots$ for matrices.

As we have mentioned in the introduction, there are basic equivalent assertions known for operator monotone functions and operator convex 
functions (cf.\cite {HP}). Namely we have

Theorem A.
For \(0 < \alpha \leq \infty \), the following assertions for a real valued continuous function f in $[0,\alpha)$ are equivalent:

(1) $f$ is operator convex and \(f(0) \leq 0\),

(2) For an operator $a$ with its spectre in $[0,\alpha)$ and a contraction $c$, 
\[ f(c^{\star}ac) \leq c^{\star}f(a)c,\]
(3) For two operators $a,b$ with their spectra in $[0,\alpha)$ and two contractions $c,d$ such that \(c^{\star}c + d^{\star}d\leq 1\) we have 
the inequality
\[ f(c^{\star}ac + d^{\star}bd) \leq c^{\star}f(a)c + d^{\star}f(b)d, \]
(4) For an operator $a$ with its spectre in $[0,\alpha)$ and a projection $p$ we have the inequality,
\[ f(pap) \leq pf(a)p \,\]
(5) The function \( g(t) = f(t)/t\) is operator monotone in the open interval $(0,\alpha)$.

In this section, we shall discuss mutual relationships of the above assertions when we restrict the property of the function $f$ at each fixed level
$n$ , that is, when $f$ and $g$ are assumed to be only $n$-matrix convex and $n$- matrix monotone. We regard the problem as the problem of 
double piling structure of those decreasing sequences $\{P_n(I)\}$ and $\{K_n(I)\}$ down to $P_{\infty}(I)$ and $K_{\infty}(I)$ respectively. 
In this sense, standard double piling structure known for these assertions before is 
the following. We describe those implications by using the convention below. 
Namely, we say the assertions (A) and (B) is in a relation 
$m \prec n$ if (A) holds for the matrix algebra $M_m$
then (B) holds for the matrix algebra $M_n$, and write \((A)_m \prec (B)_n\).

Theorem A is proved in the following way.
\[
(1)_{2n} \prec (2)_n \prec (5)_n \prec (4)_n, (2)_{2n} \prec (3)_n \prec (4)_n, \hbox{and}\ (4)_{2n} \prec (1)_n.
\]
Therefore, those assertions become equivalent when $f$ is operator convex and $g$ is operator monotone by the piling  structure.
\vskip 2mm
Thus, the basic problem for double piling structure is to find the minimum difference of degrees between those gaped assertions.
Since however even single piling problems are clarified recentry, as we have mentioned above, in spite of a long history of monotone matrix functions
and convex matrix functions, little is known for the double piling structure except the result by Mathias (\cite {M}), which asserts that
a $2n$-monotone function in the positive half line $[0, \infty)$ becomes $n$-concave.   
 
Now in order to make our investigations more transparent we mainly concentrate our discussions to the relationship about (1), (2) and (5). In fact,
we need not say anything about (4) when $n = 1$, and for the reason choosing (2) instead of (3) we just borrow the witty expression 
in \cite{HP}," correctness must bow to applicability". 
Before going into our discussions, we state each assertion in a precise way 
but skipping the condition of the spectrum of a matrix $a$. Namely, in the interval $[0,\alpha)$ we consider the following assertions.
\vskip 2mm
(i) $f(0) \leq 0$, and $f$ is n-convex,

(ii) For each positive semidefinite element a and a contraction c in $M_n$, we have
\[
f(c^{\star}a c) \leq c^{\star}f(a)c,
\]
(iii) $g$ is n-monotone in the interval $(0,\alpha)$. 
\vskip 2mm
We shall show then the equivalency of the assertions (ii) and (iii). Hence the problem is reduced to the relationship between (i) and (iii) (or
(ii)). Namely, we have the following
 
\begin{thm}\label{Th:(ii)-(iii)}
 The assertions (ii) and (iii) are equivalent.
\end{thm}

\begin{proof}  Since the implication from (ii) to (iii) is known before, we need only show the converse implication. 
Take positive semidefinite matrix $a$ with its spectrum in $[0,\alpha)$ and a contraction $c$ in $M_n$. We may assume that $a$ is
invertible. Take a positive number $\varepsilon > 0 $. From the order relation,
\[
a^{1/2}(cc^{\star} + \varepsilon)a^{1/2}\leq (1 + \varepsilon)a
\]
We have the inequality 
\[
\frac{f(a^{1/2}(cc^{\star} + \varepsilon)a^{1/2})}{a^{1/2}(cc^{\star} + \varepsilon)a^{1/2}}\leq \frac{f((1 + \varepsilon)a)}{(1 + \varepsilon)a}.
\]
Hence producting the element \(a^{1/2}(cc^{\star} + \varepsilon)a^{1/2}\) from both sides and letting $\varepsilon$ go to zero we get the inequality
\[
a^{1/2}(cc^{\star})a^{1/2}f(a^{1/2}cc^{\star} a^{1/2})\leq a^{1/2}cc^{\star}f(a)cc^{\star}a^{1/2}.
\]
Note that here we have the identity,
\[ c^{\star}a^{1/2}f(a^{1/2}cc^{\star}a^{1/2}) = f(c^{\star}ac)c^{\star}a^{1/2}.
\]
Therefore, the above inequality comes to  the form,
\[
a^{1/2}cf(c^{\star}ac)c^{\star}a^{1/2} \leq a^{1/2} cc^{\star}f(a) cc^{\star}a^{1/2}.
\]
It follows that 
\[ cf(c^{\star}ac)c^{\star} \leq cc^{\star}f(a)cc^{\star}\]
Hence for a vector $\xi$ in the underlying space $H_n$  we have 
\[
 (f(c^{\star}ac)c^{\star}\xi, c^{\star}\xi)\leq ((c^{\star}f(a)c)c^{\star}\xi, c^{\star}\xi).
\]
Now consider the orthogonal decomposition of $H_n$
with respect to the operator $c$ such as  \(H_n = [\hbox{Range}\ c^{\star}] \oplus [\hbox{Ker}\ c]\)
and write \(\xi = \xi_1 + \xi_2\). Then,
\begin{eqnarray*}
(f(c^{\star}ac)\xi,\xi)& = &(f(c^{\star}ac)\xi_1 + f(0)\xi_2, \xi_1 + \xi_2)\\
           &= & (f(c^{\star}ac)\xi_1,\xi_1) + (f(c^{\star}ac)\xi_1,\xi_2) + f(0)\Vert \xi_2 \Vert^2\\
           & = & (f(c^{\star}ac)\xi_1,\xi_1) + f(0)\Vert \xi_2\Vert^2 \\
           &\leq& (f(c^{\star}ac)\xi_1,\xi_1) \\
           &\leq& (c^{\star}f(a)c\xi_1,\xi_1)\\
           &= & (c^{\star}f(a)c\xi_,\xi ).
\end{eqnarray*}

Thus, \(f(c^{\star}ac) \leq c^{\star}f(a)c\).

In the above compitation, we have used the fact that \(f(0)\leq 0\), which is derived from the monotonicity of $g(t)$.
For, if $g(t)$ is monotone increasing we have the inequality \(f(t)\leq\frac{ f(t_0)}{t_0} t\) for every \(0 < t\leq t_0\).
This completes the proof.
\end{proof}

\vskip 2mm

We shall discuss next the gap between (i) and (iii). In the proof we need the concept of divided differences .
For a sufficientry smooth function $f(t)$ we denote its n-th divided difference 
for n-tuple of points $\{t_1,t_2,\ldots,t_n\}$ deffined as, when they are all different,
\[
[t_1,t_2]_f = \frac{f(t_1) - f(t_2)}{t_1 - t_2} ,\mbox{and inductively}
\]
\[
[t_1,t_2,\ldots,t_n]_f = \frac{[t_1,t_2,\ldots,t_{n-1}]_f - [t_2,t_3,\ldots,t_n]_f}{t_1 - t_n}.
\]
And when some of them coincides such as $t_1 = t_2 $ and so on, we put as 
\[[t_1,t_1]_f = f'(t_1)\quad \mbox { and inductively}
\]
\[
 [t_1,t_1,\ldots,t_1]_f = \frac{f^{(n)}(t_1)}{n!}.
\]

When there appears no confusion we often skip the refering function $f$. We notice here the most important property of divided differences
is that it is 
free from permutations of $\{t_1,t_2,\ldots,t_n\}$.

\begin{thm} The assertion (i) implies that $g(t)$ is $n - 1$-monotone in $(0,\alpha)$.
\end{thm}

The theorem shows that the gap from (i) to (iii) as well as (ii) is at most one, that is ,\((i)_n\prec (iii)_{n-1}\).

This improves an usual known gap $(i)_{2n} \prec (ii)_n$.


To prove this proposition we use the following simple observation.


\begin{lem}\label{simpleobservation}
Let $f$ be a function on $[0, \alpha)$ with $f(0) \leq 0$ and let $h(t) = f(t) - f(0)$. 
Then 
\begin{enumerate}
\item[$(1)$]
$f$ is n-convex if and only if $h$ is n-convex.
\item[$(2)$]
If $k(t) = \frac{h(t)}{t}$ is n-monotone, then $g(t) = \frac{f(t)}{t}$ is n-monotone.
\end{enumerate}
\end{lem}

\begin {proof}
$(1):$ It is a well-known fact.

$(2):$
 Since $\frac{h(t)}{t}$ is monotone,
\[
(f(a) - f(0))a^{-1}\leq (f(b) - f(0))b^{-1}.
\]
Hence we have
\[g(a)\leq g(b) + f(0)(a^{-1}  - b^{-1})\leq g(b)
\]
because $f(0)\leq 0$. 

This implies that
we conclude that $g$ is 2-monotone.
\end{proof}

{\bf Proof of Theorem 3.2:}  We may assume as in the proof of Theorem 1 that $f(t)$ is twice continuously differentiable. Now take a point $s$ and
fix. We consider the function \(h_s(t) = [t,s]_f\), then for two points $\{r_1,r_2\}$ we have
\[
[r_1,r_2,s]_f = [r_1,s,r_2]_f = [r_1,r_2]_{h_s}
\]
Let $\{t_1,t_2,\ldots,t_{n-1}\}$ be an arbitrary $n-1$ tuple of points in the interval $(0,\alpha)$. Then by the characterization theorem
of $n$-convexity \cite[Theorem 6.6.52 (1)]{HJ}, we see that the matrix $([t_i,t_j,s])$ is positive semidefinite for $n$-points
$\{t_1.t_2,\ldots,t_{n-1},s\}$. Hence its submatix $([t_i,t_j]_{h_s})$ is positive semidefinite, which means by \cite[Theorem 6.6.36 (1)]{HJ}
that the function $h_s(t)$ is a monotone function of degree $n-1$. Thus, in particular, \(h_0(t)= \frac{f(t) - f(0)}{t}\) becomes $n-1$-monotone.
From Lemma~\ref{simpleobservation} we have conclusion.
\hfill  \qed

\vskip 1mm

\begin{rmk}
 In connection with this theorem it would be important to note that for a finite interval we can never get the result of
Mathias' type mentioned before. In fact, in such an interval for any $2n$ we can always find a $2n$-monotone and $2n$-convex polynomial $f(t)$
by \cite[Proposition 1.3]{HT}. Therefore, if $f(t)$ become $n$-concave it had to be a constant. 
\end{rmk}

Now whether there exists an exact gap from (i) to (iii) we confirm first the following observation. Though it is almost trivial, we state 
it as a proposition for completeness sake of 
our arguments.

\begin{prp}
For n = 1, the assertion (i) implies (iii) but the converse does not hold. 
\end{prp} 

\begin{proof}  We only mention about converse. In fact, for the function \(f(t) = - t^3 + 2t^2 - t\) we see that \(g(t) = - (t - 1)^2\) is monotone
increasing in the interval $(0,1)$ but $f$ is not convex in $[0,t)$. The other case for the interval $[0,\alpha)$ is simply a consequence 
of composition function by
$f$ and the transferring function from $[0,\alpha)$ to $[0,1)$, and this holds even in the case of the positive half line. By Theorem~\ref{Th:(ii)-(iii)}
we need not discuss about (ii).
\end{proof}

\vskip 2mm

We have been however unable to decide even in the case $n = 2$ whether (i) implies (ii) or not although we can easily find a function
$f(t)$ which is not 2-convex but $g(t)$ is 2-monotone in $(0,\alpha)$.  On the other hand, we notice that there are abundance of examples of
2-convex functions in those intervals for which their associated functions are also 2-monotone. 
In fact for instance, we can show the following.

\begin{prp}\label{example(i)to(ii)} If $f(t)$ is a 2-convex polynomial of the degree not greater than 5 in $[0,\alpha)$ with $f(0)\leq 0 $, 
then $g(t)$ is 2-monotone in $(0,\alpha)$.
\end{prp}

\begin{proof}

From Lemma~\ref{simpleobservation}  we may assume that 
$f(t) = a_1t + a_2t^2 + a_3t^3 + a_4t^4 + a_5t^5$. 

Suppose that $f$ is 2-convex on $[0,\alpha)$. 
Then 
$$(\frac{f^{(i+j)}(0)}{(i+j)!}) =
\left(\begin{array}{cc}
        a_2&a_3\\
        a_3&a_4
\end{array}
\right)
$$ is positive semi-definite by \cite{Kr}. 
That is, $a_2 \geq 0$, $a_4 \geq 0$, and $a_2a_4 - a_3^2 \geq 0$.  

Then we have 
\begin{align*}
|(\frac{g^{(i+j-1)}(t)}{(i+j-1)!})| - \frac{4}{5}|(\frac{f^{(i+j)}(t)}{(i+j)!})| 
&= \frac{1}{5}(a_2a_4 - a_3^2) + 2t^2(a_4^2 + 6a_4a_5t + 10a_5^2t^2)\\
&= \frac{1}{5}(a_2a_4 - a_3^2) + 
2t^2\{10a_5^2(t + \frac{3}{10}\frac{a_4}{a_5})^2 + \frac{1}{10}a_4^2\} \geq 0,
\end{align*}
for any $t \geq 0$, where $|\quad |$ means the determinant of a given matirix.
Hence 
$|(\frac{g^{(i+j-1)}(t)}{(i+j-1)!})| \geq 0$ for $t \in [0, \alpha)$.

On the contrary,


\begin{align*}
g^{(1)}(t) - \frac{1}{5}f^{(2)}(t) 
&= a_2 + 2a_3t + 3a_4t^2 + 4a_5t^3 - \frac{2}{5}(a_2 + 3a_3t + 6a_4t^2 + 10a_5t^3)\\
&= \frac{3}{5}a_2 + \frac{4}{5}a_3t + \frac{3}{5}a_4t^2\\
&= \frac{3}{5}a_2 + \frac{3a_4}{5}(t + \frac{2}{3a_4}a_3)^2 - \frac{4}{15a_4}a_3^2 \ (\hbox{if}\ a_4\not= 0)\\
&\geq \frac{3}{5}a_2 + \frac{3a_4}{5}(t+ \frac{2}{3a_4}a_3)^2 - \frac{4}{15}a_2\ (a_2a_4 \geq a_3^2)\\
&= \frac{1}{3}a_2 + \frac{3a_4}{5}(t+ \frac{2}{3a_4}a_3)^2 \\
&\geq 0
\end{align*}
for $t \geq 0$. If $a_4 = 0$, then $a_3 = 0$. Hence we have 
$$
g^{(1)}(t) - \frac{1}{5}f^{(2)}(t) = \frac{3}{5}a_2 \geq 0.
$$
In any case we  have 
\begin{align*}
g^{(1)}(t) \geq \frac{1}{5}f^{(2)}(t) \geq 0
\end{align*}
for $t \in [0, \alpha)$. 
Similarily, we have 
\begin{align*}
\frac{g^{(3)}(t)}{3!} \geq \frac{4}{5}\frac{f^{(4)}(t)}{4!} \geq 0
\end{align*}
for $t \in [0, \alpha)$.

The above argument implies that the matrix $(\frac{g^{(i+j-1)}(t)}{(i+j-1)!})$ is 
positive semi-definite on $[0, \alpha)$. 
Therefore, we conclude that $g$ is 2-monotone on $[0, \alpha)$ 
by \cite[VIII Theorem V]{Dg}
\end{proof}

\vskip 5mm

From this proposition, we see that for $n = 2$ either the assertion (iii) does not necessarily imply the assertion (i).


In this direction, we have another result.

\begin{prp}
If $f(t)$ is 2-convex in $[0,\alpha)$, then the indefinite integral of $g(t)$ becomes also 2-convex in $(0,\alpha)$.
\end{prp}

\begin{proof} By applying the regulariation procedure (cf.\cite [chap.1.4]{Dg}) we may assume that $f$ is in the class $C^4$.
We first notice that
\[
f^{(k)}(t) = tg^{(k)}(t) + kg^{(k-1)}(t) \quad \mbox{ for \(1 \leq k \leq 4\)},
\]
which implies the relations
\[
t^{(k-1)}f^{(k)}(t) = (t^{(k)}g^{(k-1)}(t))' \quad \mbox{ for \(2\leq k \leq 4\)}.
\]
It follows that the matrix 
\[
K_2(f;t) = 
\left(
\begin{array}{cc}
\frac{1}{2}f^{(2)}(t)& \frac{1}{6}f^{(3)}(t)\\
\frac{1}{6}f^{(3)}(t)& \frac{1}{24}f^{(4)}(t)
\end {array}
\right)
\]
is positive semidefinite. Therefore both derivatives $f^{(2)}$ and $f^{(4)}$ are non-negative, and we have the inequality derived
from the determinant of the above matrix,
\[
\frac{1}{4}f^{(2)}(t)f^{(4)}(t) - \frac{1}{3}(f^{(3)}(t))^2 \geq 0.
\]
Hence,
\[  \frac{1}{4}(t^2g'(t))'(t^4 g^{(3)}(t))' - \frac{1}{3}((t^3 g^{(2)}(t))')^2 \geq 0.
\]
Therefore, we see that
\[\left(
\begin{array}{cc}
\frac{1}{2}(t^2 g'(t))'& \frac{1}{6}(t^2g''(t))'\\
\frac{1}{6}(t^3g''(t))'& \frac{1}{24}(t^4g'''(t))
\end{array}
\right)
\geq 0 \quad \mbox { for every $t$ in the interval}.
\]
Thus, integrating this matrix from $s$ to $t$ we assert that
\[
\left(
\begin{array}{cc}
\frac{1}{2}t^2g'(t)& \frac{1}{6}t^3 g''(t)\\
\frac{1}{6}t^3 g''(t) & \frac{1}{24}t^4 g '''(t)
\end{array}
\right)
- 
\left(
\begin{array}{cc}
\frac{1}{2}s^2 g'(s) & \frac{1}{6} s^3 g''(s)\\
\frac{1}{6}s^3 g''(s)& \frac{1}{24}s^4 g'''(s)
\end{array}
\right)
\geq 0.
\]
Now consider the limit matrix of the second member when $s$ goes to $0$. Using relations between those derivatives $f^{(k)}$ and $g^{(k)}$ 
mentioned at first, we see that the limit matrix has the form,
\[\left(
\begin{array}{cc}
-\frac{1}{2}f(0)& \frac{1}{3}f(0)\\
\frac{1}{3}f(0) & - \frac{1}{4}f(0)
\end{array}
\right)\],
which is obviously positive semi-definite because \(f(0)\leq 0\). 
It follows that the matrix
\[\left(
\begin{array}{cc}
\frac{1}{2}t^2g'(t)& \frac{1}{6}t^3g''(t)\\
\frac{1}{6}g''(t)& \frac{1}{24}t^4g'''(t)
\end{array}
\right)\]
is positive semidefinite. We have here the identity
\[\left(
\begin{array}{cc}
\frac{1}{2}g'(t)& \frac{1}{6}g''(t)\\
\frac{1}{6}g''(t)& \frac{1}{24}g'''(t)
\end{array}\right)
=
\left(
\begin{array}{cc}
\frac{1}{t^2}& \frac{1}{t^3}\\
\frac{1}{t^3} & \frac{1}{t^4} 
\end{array}
\right)
\circ
\left(
\begin{array}{cc}
\frac{1}{2}t^2g'(t)& \frac{1}{6}t^3 g''(t)\\
\frac{1}{6}t^3 g''(t) & \frac{1}{24}t^4 g '''(t)
\end{array}
\right)
\]
where $\circ$ means the Hadmard product. Since the Hadmard product of positive semidefinite matrices becomes
positive semidefinite we can conclude that the matrix
\[\left (
\begin{array}{cc}
\frac{1}{2}g'(t)& \frac{1}{6}g''(t)\\
\frac{1}{6}g''(t)& \frac{1}{24}g'''(t)
\end{array}\right)
\]
is positive semi-definite on $(0, \alpha)$.

By the characterization of the 2-convexity (\cite [Theorem 2.3]{HT}), 
we obtain the conclusion. This completes the proof.
\end{proof}

\end{document}